\documentclass{article}
\usepackage{amssymb}
\usepackage{amsfonts}
\usepackage{amsmath}

\newtheorem{theorem}{Theorem}

\newtheorem{example}[theorem]{Example}

\newtheorem{lemma}[theorem]{Lemma}

\newtheorem{remark}[theorem]{Remark}

\newenvironment{proof}[1][Proof]{\noindent\textbf{#1.} }{\ \rule{0.5em}{0.5em}}
\input{tcilatex}
\begin{document}

\title{Existence and multiplicity results for boundary value problems
connected with the discrete $p\left( \cdot \right) -$Laplacian on weighted
finite graphs}
\author{Marek Galewski and Renata Wieteska}
\date{}
\maketitle

\begin{abstract}
We use the direct variational method, the Ekeland variational principle, the
mountain pass geometry and Karush-Kuhn-Tucker theorem in order to
investigate existence and multiplicity results for boundary value problems
connected with the discrete $p\left( \cdot \right) -$Laplacian on weighted
finite graphs. Several auxiliary inequalities for the discrete $p\left(
\cdot \right) -$Laplacian on finite graphs are also derived. Positive
solutions are considered.
\end{abstract}


\section{Introduction}

In this note we will consider the following boundary value problem, namely 
\begin{equation}
\left\{ 
\begin{array}{l}
-\Delta _{p(x),\omega }u(x)+q(x)\left\vert u(x)\right\vert
^{p(x)-2}u(x)=\lambda f(x,u(x)),\text{ \ }x\in S,\bigskip \\ 
u(x)=0,\text{ \ \ \ \ \ \ \ \ \ \ \ \ \ \ \ \ \ \ \ \ \ \ \ \ \ \ \ \ \ \ \
\ \ \ \ \ \ \ \ \ \ \ \ \ \ \ \ \ \ \ \ \ \ \ \ }x\in \partial S.%
\end{array}%
\right.  \label{uklad}
\end{equation}%
where $\overline{S}=(S\cup \partial S,V)$ is a simple, connected, undirected
and weighted graph with two finite, disjoint and nonempty sets $S$ and $%
\partial S$ of vertices, called interior and boundary, respectively, and
with a set $V$ of unordered pairs of distinct elements of $S\cup \partial S$
whose elements are called edges, $\omega :\overline{S}\times \overline{S}%
\rightarrow \lbrack 0,+\infty )$ is a weight on a graph $\overline{S}$, $u:%
\overline{S}\rightarrow \mathbb{R}$, $q:S\rightarrow \mathbb{(}0,+\infty )$, 
$p:\overline{S}\rightarrow \mathbb{[}2,+\infty )$ are discrete functions, $%
f: $ $S\times \mathbb{R}\rightarrow \mathbb{R}$ is a continuous function, $%
\Delta _{p(\cdot ),\omega }$ is the discrete $p(\cdot )-$Laplacian defined
on a graph and $\lambda $ is a real positive parameter. The continuity of a
function $f$ means that for any fixed $x\in S$ the function $f\left( x,\cdot
\right) $ is continuous.

The aim of this paper is to investigate the existence and multiplicity of
positive solutions for problem (\ref{uklad}) applying mainly variational
methods, such as the direct variational method, the Ekeland variational
principle, the mountain pass geometry and Karush-Kuhn-Tucker theorem.
Several auxiliary inequalities for the discrete $p\left( \cdot \right) -$%
Laplacian on finite graphs are also derived. Positive solutions are the only
one that have physical meaning. That is why only these are considered.

We note that there is a big difference between discrete problems with the $%
p- $Laplacian and their graph counterparts especially in the case when the
graph is weighted which is the case of this paper. For the discrete problem,
the potential of the $p-$Laplacian, the so called isotropic case, is
coercive and has well recognized relations with any norm we can choose on
the underlying space. For the $p-$Laplacian on the weighted graph there are
no such simple relations contrary to the case of problems on non-weighted
graphs which behave almost like their discrete counterparts. Direct
calculations can be performed in order to ascertain about that. The above
mentioned reasons require considering the term $q(x)\left\vert
u(x)\right\vert ^{p-2}u(x)$ which is crucial if one wishes to apply critical
point theory and variational methods. The relations between this term and
the nonlinearity will allow us to apply classical variational tools
mentioned further in Section \ref{SecPrelimResults}.

We would like to underline that to the best of our knowledge, the discrete $%
p(\cdot )-$Laplacian on finite graphs have not been considered yet. This
means that we had to investigate the problem in a detailed manner which
involves derivation of many auxiliary inequalities that are necessary for
having the relation between the norm and potential of the graph $p\left(
\cdot \right) -$Laplacian, since now we work in the anisotropic case. The
results in the literature cover only the case of the $p-$Laplacian on
graphs, see \cite{Park-Chung, Park}, where however other methods are
applied. Thus our results are new also in the context of constant $p.$
Problems with the graph $p\left( \cdot \right) -$Laplacian called
anisotropic boundary value problems are known to be mathematical models of
various phenomena arising in the study of elastic mechanics (see \cite{B}),
electrorheological fluids (see \cite{A}) or image restoration (see \cite{C}%
). Variational continuous anisotropic problems have been started by Fan and
Zhang in \cite{FanExist1} and later considered by many methods and authors -
see \cite{hasto} for an extensive survey of such boundary value problems. In
the discrete setting see for example \cite{BerJebSer}, \cite%
{IannizottoRadulescu}, \cite{MARCU}, \cite{serban} for the most recent
results. For a background on variational methods we refer to \cite{KRV,
struwe} while for a background on difference equations to \cite{agarwalBOOK}.

We would like to note that we improve here some the existence and
multiplicity results obtained in \cite{art4} for a discrete boundary value
problem, see for example Theorem \ref{Ekeland} which in \cite{art4} requires
some assumption on the growth of nonlinearity which is not assumed here. As
concerns Theorem \ref{The_one_solution} one can derive its counterpart for
the problem considered in \cite{art4}.

The paper is organized as follows. Firstly, after providing basic
information about the graph theoretic notions, we recall the fundamental
tools from critical point theory, which cover the Weierstrass Theorem, the
Mountain Pass Lemma, the Ekeland Variational Principle and also
Karush-Kuhn-Tucker Theorem. Next we provide several inequalities useful in
variational investigations of our problem. Then we give a variational
formulation of the considered problem. The existence of positive solutions
is investigated in the last section. In the first step we establish
conditions under which we can obtain the existence of at least one positive
solution. We apply the direct variational method and Ekeland variational
principle. In the second step we are interested in the multiplicity of
positive solutions. Using the mountain pass technique both with the Ekeland
variational principle and Karush-Kuhn-Tucker conditions we obtain the
existence of at least two distinct positive solutions. The examples are also
provided.

\section{Preliminary results\label{SecPrelimResults}}

In this section we provide some tools which are used throughout the paper.
We start with providing basic information about the graph theoretic notions.

Let $\overline{S}=(S\cup \partial S,V)$ be a simple, connected, undirected
and weighted graph. A weight on a graph $\overline{S}$ is a function $\omega
:\overline{S}\times \overline{S}\rightarrow \lbrack 0,+\infty )$ satisfying%
\newline
$(i)$ $\ \omega (x,y)=\omega (y,x)$ if $\{x,y\}\in V;$\newline
$(ii)$ $\omega (x,y)=0$ if and only if $\{x,y\}\notin V.$

Let $u:\overline{S}\rightarrow \mathbb{R}$ and $p:\overline{S}\rightarrow 
\mathbb{(}1,+\infty ).$ The $p(\cdot )-$gradient $\nabla _{p(\cdot ),\omega
} $ of the function $u$ is defined by%
\begin{equation*}
\nabla _{p(x),\omega }u(x):=\left( D_{p(x),\omega ,y}u(x)\right) _{y\in 
\overline{S}}\text{ \ \ for all }x\in \overline{S},
\end{equation*}%
where%
\begin{equation*}
D_{p(x),\omega ,y}u(x):=\left\vert u(y)-u(x)\right\vert ^{p(x)-2}(u(y)-u(x))%
\sqrt{\omega (x,y)}\text{ \ for all }x\in \overline{S}
\end{equation*}%
there is the $p(\cdot )-$directional derivative\ of the function\ $u\mathbb{%
\ }$in the direction $y.$ In case of $p(\cdot )\equiv 2$ we write $\nabla
_{\omega }.$\newline
The discrete $p(\cdot )-$Laplacian $\Delta _{p(\cdot ),\omega }$ of the
function $u$ is defined by%
\begin{equation*}
\Delta _{p(x),\omega }u(x):=\underset{y\in \overline{S}}{\tsum }\left\vert
u(y)-u(x)\right\vert ^{p(x)-2}(u(y)-u(x))\omega (x,y)\text{ \ for all }x\in 
\overline{S}.
\end{equation*}%
The integration of the function $u$ on a graph $\overline{S}$ is defined by%
\begin{equation*}
\underset{\overline{S}}{\tint }u:=\underset{x\in \overline{S}}{\tsum }u(x).
\end{equation*}%
For any pair of functions $u,v:\overline{S}\rightarrow \mathbb{R}$ we have
by direct calculation (see \cite{Chung-Park, Kim})%
\begin{equation}
2\underset{\overline{S}}{\tint }(-\Delta _{p,\omega }u)v=\underset{\overline{%
S}}{\tint }\nabla _{p,\omega }u\circ \nabla _{\omega }v.  \label{2calki}
\end{equation}

Now we recall the fundamental tools from critical point theory. Let $%
(E,\left\Vert \cdot \right\Vert )$ be a real Banach space and let $%
J:E\rightarrow \mathbb{R}$. We say that the functional $J$ is coercive if $%
\lim_{\left\Vert u\right\Vert \rightarrow \infty }J(u)=+\infty $ and
anticoercive if $\lim_{\left\Vert u\right\Vert \rightarrow \infty
}J(u)=-\infty .$

\begin{theorem}
\cite{ma} \label{mbrw}(Weierstrass Theorem) Let $E$ be a reflexive Banach
space. If a functional $J\in C^{1}(E,\mathbb{R})$\ is weakly lower
semi-continuous and coercive then there exists $\overset{\_}{x}\in E$\ such
that $\underset{x\in E}{\inf }J(x)=J(\overset{\_}{x})$ and $\overset{\_}{x}$%
\ is also a critical point of $J$, i.e. $J^{^{\prime }}(\overset{\_}{x})=0.$%
Moreover, if $J$ is strictly convex, then a critical point is unique.
\end{theorem}

We say (\cite{mp}) that a continuously differentiable functional $J$\
defined on $E$ satisfies the Palais-Smale condition if every sequence $%
\{u_{n}\}$\ in $E$ such that 
\begin{equation*}
\{J(u_{n})\}\ \text{is bounded and }J^{^{\prime }}(u_{n})\longrightarrow 0%
\text{ in }E^{\ast }\text{ as }n\longrightarrow \infty
\end{equation*}%
has a convergent subsequence.

In this paper we apply the following version of the mountain pass lemma.

\begin{lemma}
\cite{mp}\label{lem2}(Mountain Pass Lemma) Let $E$ be a real Banach space
and let $J\in C^{1}(E,\mathbb{R})$ satisfy the Palais-Smale condition.
Assume that there exist $u_{0},u_{1}\in E$ and a bounded open neighborhood $%
\Omega $ of $u_{0}$ such that $u_{1}\notin \overline{\Omega }$ and 
\begin{equation*}
\max \{J(u_{0}),J(u_{1})\}<\inf_{u\in \partial \Omega }J(u).
\end{equation*}%
Let 
\begin{equation*}
\Gamma =\{\gamma \in C([0,1],E):\gamma (0)=u_{0},\text{ }\gamma (1)=u_{1}\}
\end{equation*}%
and 
\begin{equation*}
c=\inf_{\gamma \in \Gamma }\max_{t\in \lbrack 0,1]}J(\gamma (t)).
\end{equation*}%
Then $c$ is a critical value of $J$; that is, there exists $u^{\star }\in E$
such that $J^{\prime }(u^{\star })=0$ and $J(u^{\star })=c$, where $c>\max
\{J(u_{0}),J(u_{1})\}$.
\end{lemma}

We also apply the weak form of \ Ekeland's variational principle, namely

\begin{theorem}
\cite{Elel}\label{Ekeland copy(1)} (Ekeland's Variational Principle - weak
form) Let $(X,d)$ be a complete metric space. Let $\Phi :X\mathbb{%
\rightarrow R\cup \{+\infty \}}$ be lower semicontinuous and bounded from
below. Then given $\varepsilon >0$ there exists $u_{\varepsilon }\in X$ such
that:\newline
\newline
(1) $\Phi (u_{\varepsilon })\leq \inf_{u\in X}\Phi (u)+\varepsilon ,$\newline
\newline
(2) $\Phi (u_{\varepsilon })<\Phi (u)+\varepsilon d(u,u_{\varepsilon })$ \
for all $u\in X$ with $u\neq u_{\varepsilon }.$
\end{theorem}

Finally, let us recall Karush-Kuhn-Tucker conditions

\begin{theorem}
\cite{Gir}\label{KKT-THEO}\textit{(Karush-Kuhn-Tucker Theorem)} Let $E$%
\textit{\ be a Banach space and }$J_{0}$,...,$J_{n}$ be functionals on $E.$%
\textit{\ Let }$x_{0}$ be a solution of the problem%
\begin{equation*}
\left\{ 
\begin{array}{c}
\underset{x\in E}{\min }J_{0}(x);\text{ \ \ \ \ \ \ \ \ \ \ \ \ \ \ \ \ }%
\bigskip \\ 
J_{i}(x)\leq 0,\text{ }i=1,...,n,%
\end{array}%
\right.
\end{equation*}%
and assume that the functionals $J_{0},...,J_{n}$ are Fr\'{e}chet
differentiable at $x_{0}.$ Then there exist nonnegative real numbers $%
\lambda _{0},...,\lambda _{n}$, not all zero, such that%
\begin{equation*}
\lambda _{i}J_{i}(x_{0})=0,\text{ }i=1,...,n,
\end{equation*}%
and%
\begin{equation*}
\lambda _{0}J_{0}^{\prime }(x_{0})+\lambda _{1}J_{1}^{\prime
}(x_{0})+...+\lambda _{n}J_{n}^{\prime }(x_{0})=0\text{.}
\end{equation*}%
If the set $\left\{ x:J_{i}(x)\leq 0,\text{ }i=1,...,n\right\} $ has a
non-empty interior (Slater's condition) then we can fix $\lambda _{0}=1$.
\end{theorem}

In this paper we also use ideas in \cite{BerJebMawhin}, where the Authors
proved that if $X\subset E$ is an open set, the functional $J$ satisfies the
Palais-Smale condition and 
\begin{equation}
\inf_{x\in \overline{X}}J\left( x\right) <\inf_{x\in \partial X}J\left(
x\right)  \label{J_domX<J_brzegX}
\end{equation}%
then there exists some $x_{0}\in E$ which is a critical point to $J$ such
that%
\begin{equation*}
\inf_{x\in \overline{X}}J\left( x\right) =J\left( x_{0}\right) .
\end{equation*}

The multiplicity result in \cite{BerJebMawhin} is also obtained with the aid
of the mountain pass lemma in the following context. If $X$ is an open ball
centered at $0$ with radius $r>0,$ the functional $J$ satisfies the
Palais-Smale condition and $J\left( 0\right) =0$ then there exists an
element $e\in E\backslash X$ such that $J\left( e\right) \leq 0$. If
additionally%
\begin{equation*}
-\infty <\inf_{x\in \overline{X}}J\left( x\right) <0<\inf_{x\in \partial
X}J\left( x\right) ,
\end{equation*}%
then $J$ has two critical points.

By using the above mentioned methods the Author in \cite{serban2} obtains
the existence of at least two non-zero solutions for some periodic and
Neumann problems with the discrete $p(\cdot )-$Laplacian.

\section{Auxiliary inequalities\label{SecIneq}}

Recall that in this paper we examine the existence of solutions for problem
with Dirichlet boundary conditions. Therefore let us define the space $A$ in
which the problem will be considered as follows%
\begin{equation*}
A:=\{u:\overline{S}\rightarrow \mathbb{R}:u(x)=0\text{ \ for all \ }x\in
\partial S\}.
\end{equation*}%
Then $A$ is a finite dimensional Euclidean space provided with the norm
given by 
\begin{equation}
\left\Vert u\right\Vert :=\left( \underset{\overline{S}}{\tint }\left\vert
u\right\vert ^{2}\right) ^{\frac{1}{2}}  \label{norm}
\end{equation}%
and with naturally associated scalar product. The dimension of $A$ is $%
\left\vert S\right\vert $.

Let us also introduce a notation used throughout the paper, namely%
\begin{equation*}
\begin{array}{l}
p^{-}:=\underset{x\in S}{\min }\text{ }p(x);\text{ \ \ }p^{+}:=\underset{%
x\in S}{\max }\text{ }p(x);\bigskip \\ 
\overline{p}^{-}:=\underset{x\in \overline{S}}{\min }\text{ }p(x);\text{ \ \ 
}\overline{p}^{+}:=\underset{x\in \overline{S}}{\max }\text{ }p(x);\bigskip
\\ 
\text{ \ \ \ \ \ \ \ }\overline{\omega }^{+}:=\underset{(x,y)\in \overline{S}%
\times \overline{S}}{\max }\omega (x,y).%
\end{array}%
\end{equation*}

On the space $A$ the following inequalities hold.

\begin{lemma}
\label{inequalities}\ \newline
\textbf{(a.1)} For every $u\in A$ and for every $m\geq 1$ we have%
\begin{equation*}
\dsum\limits_{x\in S}|u(x)|^{m}\leq \left\vert S\right\vert \left\Vert
u\right\Vert ^{m}.\text{ \ \ \ \ \ \ \ \ \ \ \ \ \ \ \ \ \ \ \ \ \ \ \ \ \ \
\ \ \ \ \ \ \ \ \ \ \ \ \ \ \ \ \ \ \ \ \ \ }
\end{equation*}%
\textbf{(a.2)} For every $u\in A$ and for every $m\geq 2$ we have%
\begin{equation*}
\dsum\limits_{x,y\in \overline{S}}\left\vert u(y)-u(x)\right\vert ^{m}\leq
2^{m}\left\vert \overline{S}\right\vert \left\vert S\right\vert \left\Vert
u\right\Vert ^{m}.\text{ \ \ \ \ \ \ \ \ \ \ \ \ \ \ \ \ \ \ \ \ \ \ \ \ \ \
\ \ \ \ \ \ }
\end{equation*}%
\textbf{(a.3)} For every $u\in A$ and for every $m\geq 2$ we have%
\begin{equation*}
\dsum\limits_{x\in S}\left\vert u(x)\right\vert ^{m}\geq 2^{-\frac{m}{2}%
}\left\vert \partial S\right\vert ^{\frac{m}{2}}\left\vert \overline{S}%
\right\vert ^{1-m}\left\Vert u\right\Vert ^{m}.\text{ \ \ \ \ \ \ \ \ \ \ \
\ \ \ \ \ \ \ \ \ \ \ \ \ \ \ \ \ \ }
\end{equation*}%
\textbf{(a.4)} For every $u\in A$ and for every $p^{-}\geq 2$ we have%
\begin{equation*}
\dsum\limits_{x\in S}\left\vert u(x)\right\vert ^{p(x)}\geq 2^{-\frac{p^{-}}{%
2}}\left\vert \partial S\right\vert ^{\frac{p^{-}}{2}}\left\vert \overline{S}%
\right\vert ^{1-p^{-}}\left\Vert u\right\Vert ^{p^{-}}-\left\vert
S\right\vert .\text{ \ \ \ \ \ \ \ \ \ \ \ \ \ \ \ \ }
\end{equation*}%
\textbf{(a.5)} For every $u\in A$ and for every $\overline{p}^{+}\geq 2$ we
have%
\begin{equation*}
\dsum\limits_{x,y\in \overline{S}}\left\vert u(y)-u(x)\right\vert
^{p(x)}\omega (x,y)\leq \overline{\omega }^{+}2^{\overline{p}^{+}}\left\vert 
\overline{S}\right\vert \left\vert S\right\vert \left\Vert u\right\Vert ^{%
\overline{p}^{+}}+\overline{\omega }^{+}\left\vert \overline{S}\right\vert
^{2}.\text{ \ \ }
\end{equation*}%
\textbf{(a.6)} For every $u\in A$ and for every $p^{+}\geq 2$ we have%
\begin{equation*}
\dsum\limits_{x\in S}\left\vert u(x)\right\vert ^{p(x)}\leq \left\vert
S\right\vert \left\Vert u\right\Vert ^{p^{+}}+\left\vert S\right\vert .\text{
\ \ \ \ \ \ \ \ \ \ \ \ \ \ \ \ \ \ \ \ \ \ \ \ \ \ \ \ \ \ \ \ \ \ \ \ \ \ }
\end{equation*}%
\textbf{(a.7)} For every $u\in A$ we have%
\begin{equation*}
\underset{x\in S}{\max }\left\vert u(x)\right\vert \leq \left\vert \overline{%
S}\right\vert ^{\frac{1}{2}}\left\Vert u\right\Vert .\text{ \ \ \ \ \ \ \ \
\ \ \ \ \ \ \ \ \ \ \ \ \ \ \ \ \ \ \ \ \ \ \ \ \ \ \ \ \ \ \ \ \ \ \ \ \ \
\ \ }
\end{equation*}
\end{lemma}

\begin{proof}
We will show that \textbf{(a.1) }holds\textbf{. }For all $x\in S$ we have%
\begin{equation*}
|u(x)|^{2}\leq \dsum\limits_{s\in S}|u(s)|^{2}=\dsum\limits_{s\in \overline{S%
}}|u(s)|^{2}.
\end{equation*}%
Thus for every $m\geq 1$ we get%
\begin{equation*}
|u(x)|^{m}\leq \left( \dsum\limits_{s\in \overline{S}}|u(s)|^{2}\right) ^{%
\frac{m}{2}},
\end{equation*}%
which leads to%
\begin{equation*}
\dsum\limits_{x\in S}\ |u(x)|^{m}\leq \left\vert S\right\vert \left\Vert
u\right\Vert ^{m}.
\end{equation*}

To see \textbf{(a.2)} first note that for every $u\in A$ and for every $m>0$
we have 
\begin{equation*}
\dsum\limits_{x,y\in \overline{S}}\left\vert u(x)\right\vert
^{m}=\dsum\limits_{y\in \overline{S}}\left( \dsum\limits_{x\in \overline{S}%
}\left\vert u(x)\right\vert ^{m}\right) =\left\vert \overline{S}\right\vert
\dsum\limits_{x\in S}\left\vert u(x)\right\vert ^{m}.\bigskip
\end{equation*}%
Recall also that for every $m\geq 2$ the following inequality hold (see \cite%
{Bre})%
\begin{equation*}
\left\vert a+b\right\vert ^{m}+\left\vert a-b\right\vert ^{m}\leq
2^{m-1}\left( \left\vert a\right\vert ^{m}+\left\vert b\right\vert
^{m}\right) \text{ for all }a,b\in \mathbb{R}\text{.}
\end{equation*}%
Thus for every $u\in A$ and for every $m\geq 2$ we have 
\begin{equation}
\dsum\limits_{x,y\in \overline{S}}|u(y)-u(x)|^{m}\leq 2^{m}\left\vert 
\overline{S}\right\vert \dsum\limits_{x\in S}\left\vert u(x)\right\vert ^{m}.
\label{ineq_a4}
\end{equation}%
In a consequence, by \textbf{(a.1)} we get \textbf{(a.2)}.

We will show that \textbf{(a.3) }holds\textbf{. }Using twice the discrete H%
\"{o}lder inequality we have%
\begin{equation*}
\begin{array}{l}
\dsum\limits_{x,y\in \overline{S}}\left\vert u(y)-u(x)\right\vert
^{2}=\dsum\limits_{x,y\in \overline{S}}\left\vert u(y)\right\vert
^{2}+\dsum\limits_{x,y\in \overline{S}}\left\vert u(x)\right\vert
^{2}-\bigskip \\ 
2\dsum\limits_{x,y\in \overline{S}}\left( u(y)u(x)\right) =2\left\vert 
\overline{S}\right\vert \left\Vert u\right\Vert ^{2}-2\dsum\limits_{x\in
S}\left( \left( \dsum\limits_{y\in S}u(y)\right) u(x)\right) \geq \bigskip
\\ 
2\left\vert \overline{S}\right\vert \left\Vert u\right\Vert ^{2}-2\left(
\dsum\limits_{x\in S}\left( \dsum\limits_{y\in S}u(y)\right) ^{2}\right) ^{%
\frac{1}{2}}\left( \dsum\limits_{x\in S}\left\vert u(x)\right\vert
^{2}\right) ^{\frac{1}{2}}=\bigskip \bigskip \\ 
2\left\vert \overline{S}\right\vert \left\Vert u\right\Vert ^{2}-2\left\vert
S\right\vert ^{\frac{1}{2}}\dsum\limits_{y\in S}u(y)\left\Vert u\right\Vert
\geq 2\left\Vert u\right\Vert ^{2}\left\vert \partial S\right\vert .%
\end{array}%
\end{equation*}%
On the other hand for every $m\geq 2$ the discrete H\"{o}lder inequality
implies%
\begin{equation*}
\begin{array}{l}
\dsum\limits_{x,y\in \overline{S}}\left\vert u(y)-u(x)\right\vert ^{2}\leq
\left( \dsum\limits_{x,y\in \overline{S}}1^{\frac{m}{m-2}}\right) ^{\frac{m-2%
}{m}}\left( \dsum\limits_{x,y\in \overline{S}}\left( \left\vert
u(y)-u(x)\right\vert ^{2}\right) ^{\frac{m}{2}}\right) ^{\frac{2}{m}%
}=\bigskip \bigskip \\ 
\left\vert \overline{S}\right\vert ^{\frac{2\left( m-2\right) }{m}}\left(
\dsum\limits_{x,y\in \overline{S}}|u(y)-u(x)|^{m}\right) ^{\frac{2}{m}%
}.\bigskip%
\end{array}%
\end{equation*}%
The above inequalities lead to%
\begin{equation*}
2\left\Vert u\right\Vert ^{2}\left\vert \partial S\right\vert \leq
\left\vert \overline{S}\right\vert ^{\frac{2\left( m-2\right) }{m}}\left(
\dsum\limits_{x,y\in \overline{S}}|u(y)-u(x)|^{m}\right) ^{\frac{2}{m}}.
\end{equation*}%
Thus for every $u\in A$ and for every $m\geq 2$ we have%
\begin{equation}
\dsum\limits_{x,y\in \overline{S}}|u(y)-u(x)|^{m}\geq 2^{\frac{m}{2}%
}\left\vert \partial S\right\vert ^{\frac{m}{2}}\left\vert \overline{S}%
\right\vert ^{2-m}\left\Vert u\right\Vert ^{m}.  \label{ineq_a3}
\end{equation}%
Combining (\ref{ineq_a4}) and \textbf{(}\ref{ineq_a3}) we get \textbf{(a.3)}.

Relation \textbf{(a.4) }is obtained by \textbf{(a.3)} as follows%
\begin{equation*}
\begin{array}{l}
\dsum\limits_{x\in S}\left\vert u(x)\right\vert ^{p(x)}\geq
\dsum\limits_{\left\{ x\in S:\left\vert u(x)\right\vert >1\right\}
}\left\vert u(x)\right\vert ^{p^{-}}+\dsum\limits_{\left\{ x\in S:\left\vert
u(x)\right\vert \leq 1\right\} }\left\vert u(x)\right\vert ^{p^{+}}=\bigskip
\bigskip \\ 
\dsum\limits_{x\in S}\left\vert u(x)\right\vert
^{p^{-}}-\dsum\limits_{\left\{ x\in S:\left\vert u(x)\right\vert \leq
1\right\} }\left( \left\vert u(x)\right\vert ^{p^{-}}-\left\vert
u(x)\right\vert ^{p^{+}}\right) \geq \bigskip \bigskip \\ 
2^{-\frac{p^{-}}{2}}\left\vert \partial S\right\vert ^{\frac{p^{-}}{2}%
}\left\vert \overline{S}\right\vert ^{1-p^{-}}\left\Vert u\right\Vert
^{p^{-}}-\left\vert S\right\vert .%
\end{array}%
\end{equation*}

Relation \textbf{(a.5) }is obtained by \textbf{(a.2)} as follows 
\begin{equation*}
\begin{array}{l}
\dsum\limits_{x,y\in \overline{S}}|u(y)-u(x)|^{p(x)}\omega (x,y)\leq 
\overline{\omega }^{+}\dsum\limits_{x,y\in \overline{S}}|u(y)-u(x)|^{p(x)}%
\leq \bigskip \\ 
\overline{\omega }^{+}\dsum\limits_{\left\{ x,y\in \overline{S}:\left\vert
u(y)-u(x)\right\vert >1\right\} }\left\vert u(y)-u(x)\right\vert ^{\overline{%
p}^{+}}+\bigskip \\ 
\overline{\omega }^{+}\dsum\limits_{\left\{ x,y\in \overline{S}:\left\vert
u(y)-u(x)\right\vert \leq 1\right\} }\left\vert u(y)-u(x)\right\vert ^{%
\overline{p}^{-}}\leq \bigskip \bigskip \\ 
\overline{\omega }^{+}\dsum\limits_{x,y\in \overline{S}}|u(y)-u(x)|^{%
\overline{p}^{+}}+\bigskip \\ 
\overline{\omega }^{+}\dsum\limits_{\left\{ x,y\in \overline{S}:\left\vert
u(y)-u(x)\right\vert \leq 1\right\} }\left( \left\vert u(y)-u(x)\right\vert
^{\overline{p}^{-}}-\left\vert u(y)-u(x)\right\vert ^{\overline{p}%
^{+}}\right) \leq \bigskip \bigskip \\ 
\overline{\omega }^{+}2^{\overline{p}^{+}}\left\vert \overline{S}\right\vert
\left\vert S\right\vert \left\Vert u\right\Vert ^{\overline{p}^{+}}+%
\overline{\omega }^{+}\left\vert \overline{S}\right\vert ^{2}.%
\end{array}%
\end{equation*}

The inequality \textbf{(a.6) }we obtain\textbf{\ }in the same manner as 
\textbf{(a.5),} using \textbf{(a.1)} instead of \textbf{(a.2).}

And finally, the discrete H\"{o}lder inequality implies \textbf{(a.7). }%
Indeed, for every $x\in S$ we have 
\begin{equation*}
\left\vert u(x)\right\vert \leq \dsum\limits_{s\in S}\left\vert
u(s)\right\vert =\dsum\limits_{s\in \overline{S}}\left\vert u(s)\right\vert
\leq \left\vert \overline{S}\right\vert ^{\frac{1}{2}}\left\Vert
u\right\Vert .
\end{equation*}%
Therefore 
\begin{equation*}
\underset{x\in S}{\max }\left\vert u(x)\right\vert \leq \left\vert \overline{%
S}\right\vert ^{\frac{1}{2}}\left\Vert u\right\Vert .
\end{equation*}%
The proof of Lemma \ref{inequalities} is complete.
\end{proof}

\section{Variational framework}

In order to study the problem considered we will start with putting in the
nonlinear term $f$\ the non-negative part of $u$ instead of $u$. Then we
obtain the following boundary value problem%
\begin{equation}
\left\{ 
\begin{array}{l}
-\Delta _{p(x),\omega }u(x)+q(x)\left\vert u(x)\right\vert
^{p(x)-2}u(x)=\lambda f(x,u_{+}(x)),\text{ \ }x\in S,\bigskip \\ 
u(x)=0,\text{ \ \ \ \ \ \ \ \ \ \ \ \ \ \ \ \ \ \ \ \ \ \ \ \ \ \ \ \ \ \ \
\ \ \ \ \ \ \ \ \ \ \ \ \ \ \ \ \ \ \ \ \ \ \ \ \ \ }x\in \partial S.%
\end{array}%
\right.  \label{uklad2+}
\end{equation}

Let us define the functional $J:A\rightarrow \mathbb{R}$ by the following
formula%
\begin{equation}
J(u)=\frac{1}{2}\underset{\overline{S}}{\int }\frac{1}{p}\nabla _{p,\omega
}u\circ \nabla _{\omega }u+\underset{S}{\int }\frac{1}{p}q\left\vert
u\right\vert ^{p}-\lambda \underset{S}{\int }F_{u_{+}},  \label{action}
\end{equation}%
where $F_{u_{+}}:S\rightarrow \mathbb{R}$ is defined by 
\begin{equation*}
F_{u_{+}}(x)=F(x,u_{+}(x)):=\overset{u_{+}(x)}{\underset{0}{\int }}f(x,s)ds%
\text{.}
\end{equation*}%
The functional $J$ can be rewritten as follows%
\begin{equation*}
\begin{array}{l}
J(u)=\frac{1}{2}\dsum\limits_{x\in \overline{S}}\left( \frac{1}{p(x)}\text{ }%
\dsum\limits_{y\in \overline{S}}\left\vert u(y)-u(x)\right\vert
^{p(x)}\omega (x,y)\right) \bigskip \\ 
+\dsum\limits_{x\in S}\frac{1}{p(x)}q(x)\left\vert u(x)\right\vert
^{p(x)}-\lambda \dsum\limits_{x\in S}F(x,u_{+}(x))%
\end{array}%
\end{equation*}%
and we will use both notations when necessary.

We will show that critical points of the functional $J$ correspond to the
solutions of problem (\ref{uklad2+}).

\begin{theorem}
\label{critical_point}\textit{The point }$u\in A$ \textit{is a critical
point to }$J$ if and only if it satisfies\ (\ref{uklad2+}).
\end{theorem}

\begin{proof}
Take an arbitrary $u\in A$. Let $\varphi :\mathbb{R}\longrightarrow \mathbb{R%
}$ be given by $\varphi (\varepsilon )=J\left( u+\varepsilon v\right) ,$
where $v\in A$ is a fixed nonzero direction. Then%
\begin{equation*}
\begin{array}{l}
\varphi ^{^{\prime }}(\varepsilon )=\frac{1}{2}\dsum\limits_{x,y\in 
\overline{S}}\left\vert (u+\varepsilon v)(y)-(u+\varepsilon v)(x)\right\vert
^{p(x)-2}\bigskip \\ 
((u+\varepsilon v)(y)-(u+\varepsilon v)(x))(v(y)-v(x))\omega (x,y)+\bigskip
\bigskip \\ 
\dsum\limits_{x\in S}q(x)\left\vert (u+\varepsilon v)(x)\right\vert
^{p(x)-2}(u+\varepsilon v)(x)v(x)-\lambda \dsum\limits_{x\in
S}f(x,(u_{+}+\varepsilon v)(x))v(x).%
\end{array}%
\end{equation*}%
Letting $\varepsilon =0$ we have%
\begin{equation}
\begin{array}{l}
\varphi ^{^{\prime }}(0)=\frac{1}{2}\dsum\limits_{x,y\in \overline{S}%
}\left\vert u(y)-u(x)\right\vert ^{p(x)-2}(u(y)-u(x))(v(y)-v(x))\omega
(x,y)+\bigskip \\ 
\dsum\limits_{x\in S}q(x)\left\vert u(x)\right\vert
^{p(x)-2}u(x)v(x)-\lambda \dsum\limits_{x\in S}f(x,u_{+}(x))v(x).%
\end{array}
\label{nr_poch}
\end{equation}%
Thus by (\ref{2calki}) 
\begin{equation}
\varphi ^{^{\prime }}(0)=\underset{\overline{S}}{\int }(-\Delta _{p,\omega
}u)v+\underset{S}{\int }q\left\vert u\right\vert ^{p-2}uv-\lambda \underset{S%
}{\int }f_{u_{+}}v.  \label{f_prim}
\end{equation}%
Let us fix $x\in S$ and let us define a function $v:\overline{S}\rightarrow 
\mathbb{R}$ by the following formula%
\begin{equation*}
v(w)=\left\{ 
\begin{array}{c}
1\text{ for }x=w \\ 
0\text{ otherwise.}%
\end{array}%
\right.
\end{equation*}%
Then we see from (\ref{f_prim}) that%
\begin{equation*}
-\Delta _{p(x),\omega }u(x)+q(x)\left\vert u(x)\right\vert
^{p(x)-2}u(x)-\lambda f(x,u_{+}(x))=0.
\end{equation*}%
Since $x\in S$ was fixed arbitrarily we get 
\begin{equation*}
-\Delta _{p(x),\omega }u(x)+q(x)\left\vert u(x)\right\vert
^{p(x)-2}u(x)-\lambda f(x,u_{+}(x))=0\text{ for all }x\in S.\text{ }
\end{equation*}%
Thus if $u\in A$ is a critical point the functional $J$ it is a solution to
problem (\ref{uklad2+}). It is easy to see that every solution to problem (%
\ref{uklad2+}) it is also a critical point to the functional $J.\bigskip $
\end{proof}

One important remark is in order as concerns the action functional $J$
(given by the formula (\ref{action})). In the discrete boundary value
problem one may take either term connected with the difference operator or
else with the nonlinearity as the leading one since all norms on a finite
dimensional space are equivalent. In our case such approach is not possibly
because of the presence of the weight $\omega $ for which we cannot derive
suitable inequalities as given in Section \ref{SecIneq} and by the fact that
we investigate the existence of positive solutions. Thus we shall use mainly
term $\underset{S}{\int }\frac{1}{p}q\left\vert u\right\vert ^{p}$ as the
leading term in our investigations.

\section{Existence of positive solutions}

In this section we will seek positive solutions to problem (\ref{uklad}). By
a positive solution to problem (\ref{uklad}) we mean such a function $u:%
\overline{S}\rightarrow \mathbb{R}$ which satisfies the given equation on $%
\overline{S}$, the boundary conditions on $\partial S$ and it has only
positive values on $S$. Positive solutions to (\ref{uklad}) are investigated
in the space $A$\ considered with the norm (\ref{norm}).

Put 
\begin{equation*}
u_{+}(x)=\max \{u(x),0\},\text{ \ \ }u_{-}(x)=\max \{-u(x),0\}\text{ \ for
all }x\in \overline{S}.
\end{equation*}%
It is easy to see that for all $x\in \overline{S}$ we have%
\begin{equation*}
\begin{array}{l}
u_{+}(x),u_{-}(x)\geq 0\text{ };\bigskip \\ 
u(x)=u_{+}(x)-u_{-}(x);\bigskip \\ 
u_{+}(x)\cdot u_{-}(x)=0;\bigskip \\ 
\left\vert u(x)\right\vert =u_{+}(x)+u_{-}(x).\bigskip \\ 
\left\vert u(x)_{+}\right\vert \leq \left\vert u(x)\right\vert .%
\end{array}%
\end{equation*}

Let us formulate an auxiliary result which plays an important role in
proving all the existence results in this section. This result shows that
any solution to (\ref{uklad2+}) is in fact a positive solution and
simultaneously it is the positive solution to (\ref{uklad}). It may be
viewed as a kind of a discrete maximum principle.$\bigskip $

Assume that

\textit{\textbf{(f.0)}} \textit{The function }$f$\textit{\ takes positive
values for all }$x\in S$\textit{\ and all }$t\geq 0.\bigskip $

\begin{lemma}
\label{positive}Assume that \textit{\textbf{(f.0) }}holds. Assume that $u\in
A$ is a solution to problem (\ref{uklad2+}). Then $u$\textit{\ has only
positive values on} $S$ and moreover $u$ is a positive solution to (\ref%
{uklad}).
\end{lemma}

\begin{proof}
A straightforward computation shows that for every$\ x,y\in \overline{S}$
the following inequality 
\begin{equation}
\left( u(y)-u(x)\right) \left( u_{-}(y)-u_{-}(x)\right) \leq 0  \label{u_u_}
\end{equation}%
holds. Indeed, 
\begin{equation*}
\begin{array}{l}
\left( u(y)-u(x)\right) \left( u_{-}(y)-u_{-}(x)\right) =\bigskip \\ 
\left( \left( u_{+}(y)-u_{+}(x)\right) -\left( u_{-}(y)-u_{-}(x)\right)
\right) \left( u_{-}(y)-u_{-}(x)\right) =\bigskip \\ 
-\left( u_{+}(y)u_{-}(x)+u_{+}(x)u_{-}(y)+\left( u_{-}(y)-u_{-}(x)\right)
^{2}\right) \leq 0.%
\end{array}%
\end{equation*}%
Assume that $u\in A$ is a solution to (\ref{uklad2+}). Equating (\ref%
{nr_poch}) to $0$ and taking $v=u_{-}$ we obtain%
\begin{equation}
\begin{array}{l}
\frac{1}{2}\dsum\limits_{x,y\in \overline{S}}\left\vert u(y)-u(x)\right\vert
^{p(x)-2}(u(y)-u(x))(u_{-}(y)-u_{-}(x))\omega (x,y)=\bigskip \\ 
\lambda \underset{x\in S}{\tsum }f(x,u_{+}(x))u_{-}(x)-\underset{x\in S}{%
\tsum }q(x)\left\vert u(x)\right\vert ^{p(x)-2}u(x)u_{-}(x).%
\end{array}
\label{eqqu}
\end{equation}%
Since $f$ and $q$ are functions with positive values only, $\lambda >0$ and
since%
\begin{equation*}
u(x)u_{-}(x)=\left( u_{+}(x)-u_{-}(x)\right)
u_{-}(x)=u_{+}(x)u_{-}(x)-\left( u_{-}(x)\right) ^{2}\leq 0
\end{equation*}%
the term on the right is non-negative. Due to (\ref{u_u_}) the term on the
left is non-positive, therefore equation (\ref{eqqu}) holds if {the both
terms are equal zero, which} leads to relation $u_{-}(x)=0$ for all $x\in S.$
Thus $u(x)=u_{+}(x)$ for all $x\in S.$ Moreover $u\left( x\right) \neq 0$
for all $x\in S$. Indeed, assume that there exists $x_{0}\in S$ such that $%
u(x_{0})=0$. Then by (\ref{uklad}) we have%
\begin{equation*}
-\underset{\{y\in \overline{S}:\text{ }y\neq x_{0}\}}{\tsum }\left\vert
u(y)\right\vert ^{p(x_{0})-1}\omega (x_{0},y)=\lambda f(x_{0},0).
\end{equation*}%
Since the term on the left is non-positive and the term on the right
positive we have a contradiction. Thus $u\left( x\right) \neq 0$ for all $%
x\in S,$ it follows that $u$ is a positive solution to (\ref{uklad}).$%
\bigskip $
\end{proof}

To show that problem (\ref{uklad}) has positive solutions we need the
following growth conditions$\bigskip $

\textit{\textbf{(f.1)}} \textit{There exist functions }$m_{1},m_{2}:S%
\rightarrow \lbrack 2,+\infty )$\textit{\ and functions }$\varphi
_{1},\varphi _{2},\psi _{1},\psi _{2}:S\rightarrow (0,+\infty )$\textit{\
such that}%
\begin{equation*}
\psi _{1}(x)+\varphi _{1}(x)t^{m_{1}(x)-1}\leq f(x,t)\leq \varphi
_{2}(x)t^{m_{2}(x)-1}+\psi _{2}(x)
\end{equation*}%
\textit{for all }$x\in S$\textit{\ and all }$t\geq 0.\bigskip $

Note that \textbf{(f.1) }implies \textbf{(f.0). }Using the definition of $F$
we get by integration$\bigskip $

\textit{\textbf{(F.1)}}\textbf{\ }\textit{For functions }$%
m_{1},m_{2}:S\rightarrow \lbrack 2,+\infty )$\textit{\ and functions }$%
\varphi _{1},\varphi _{2},\psi _{1},\psi _{2}:S\rightarrow (0,+\infty )$%
\textit{\ satisfying \textbf{(f.1) }we have}%
\begin{equation*}
\psi _{1}(x)t+\frac{\varphi _{1}(x)}{m_{1}(x)}t^{m_{1}(x)}\leq F(x,t)\leq 
\frac{\varphi _{2}(x)}{m_{2}(x)}t^{m_{2}(x)}+\psi _{2}(x)t
\end{equation*}%
\textit{for all }$x\in S$\textit{\ and all }$t\geq 0.\bigskip $

Let us introduce the following notations%
\begin{equation*}
\begin{array}{l}
q^{-}=\underset{x\in S}{\min }\text{ }q\left( x\right) ,\text{ \ \ \ \ \ }%
q^{+}=\underset{x\in S}{\max }\text{ }q\left( x\right) ,\bigskip \\ 
m_{i}^{-}=\underset{x\in S}{\min }\text{ }m\left( x\right) ,\text{ \ \ \ }%
m_{i}^{+}=\underset{x\in S}{\max }\text{ }m\left( x\right) ,\bigskip \text{ }%
i=1,2\text{,} \\ 
\varphi _{1}^{-}=\underset{x\in S}{\min }\text{ }\varphi _{1}\left( x\right)
,\text{ \ \ \ }\varphi _{2}^{+}=\underset{x\in S}{\max }\text{ }\varphi
_{2}\left( x\right) ,\text{ }\bigskip \\ 
\psi _{1}^{-}=\underset{x\in S}{\min }\text{ }\psi _{1}\left( x\right) ,%
\text{ \ \ \ }\psi _{2}^{+}=\underset{x\in S}{\max }\text{ }\psi _{2}\left(
x\right) ,\text{\ }%
\end{array}%
\end{equation*}%
where $m_{1},m_{2},\varphi _{1},\varphi _{2},\psi _{1},\psi _{2}$ are
functions defined in (\textbf{f.1).}$\bigskip $

Now we give an example to illustrate condition \textit{(\textbf{f.1).}}

\begin{example}
\label{przyklad}Let $m:S\rightarrow \lbrack 2,+\infty )$ and $\varphi ,\psi
:S\rightarrow (0,+\infty ).$ The function $f:S\times \mathbb{R\rightarrow R}$
given by the formula 
\begin{equation*}
f(x,t)=(t+1)^{1-e^{-t^{2}}+m(x)}\left( \frac{2}{\pi }\arctan t+\varphi
(x)\right) +\left\vert \sin t\right\vert +\psi (x)+1
\end{equation*}%
is a \textit{continuous function with only positive values for all }$x\in S$%
\textit{\ and all }$t\geq 0$ and%
\begin{equation*}
\begin{array}{l}
1+\psi (x)+t^{m(x)}\varphi (x)\leq f(x,t)\leq \bigskip \\ 
2^{m(x)}t^{1+m(x)}\left( 1+\varphi (x)\right) +2^{m(x)}\left( 1+\varphi
(x)\right) +\psi (x)+2.%
\end{array}%
\end{equation*}%
\textit{for all }$x\in S$\textit{\ and all }$t\geq 0,$ so the growth
conditions are satisfied with $m_{1}(x)=m(x),$ $m_{2}(x)=m(x)+1,$ $\varphi
_{1}(x)=\varphi (x),$ $\varphi _{2}(x)=2^{m(x)}\left( 1+\varphi (x)\right) ,$
$\psi _{1}(x)=\psi (x)+1$ and $\psi _{2}(x)=\psi (x)+2$.
\end{example}

We will investigate the existence of positive solutions applying different
methods, since depending on a relation between functions $m_{1},m_{2}$ and $%
p $ the functional $J$ has different properties.

\subsection{Results by the direct variational approach}

We start with a case $m_{2}^{+}<p^{-}.$ Then for all $\lambda >0$ the
functional $J$ is coercive and we can apply the direct variational method,
Theorem \ref{mbrw}. The case $m_{2}^{+}=p^{-}$ is also undertaken, but in
this case there is some restrictions on the parameter $\lambda .$

\begin{theorem}
\label{TheoDirect}Let $m_{2}^{+}<p^{-}.$ Assume that condition \textbf{(f.1)}
is satisfied. Then for all $\lambda >0$ problem (\ref{uklad}) has at least
one positive solution.
\end{theorem}

\begin{proof}
It suffices to show that the functional $J$ is coercive on the set $A$ so
that to apply Theorem \ref{mbrw}. By \textbf{(a.7) }we have\textbf{\ }%
\begin{equation}
\dsum\limits_{x\in S}u_{+}(x)\leq \dsum\limits_{x\in S}\left\vert
u_{+}(x)\right\vert \leq \dsum\limits_{x\in S}\left\vert u(x)\right\vert
\leq \left\vert S\right\vert \underset{x\in S}{\max }\left\vert
u(x)\right\vert \leq \left\vert S\right\vert \left\vert \overline{S}%
\right\vert ^{\frac{1}{2}}\left\Vert u\right\Vert .  \label{u+norma}
\end{equation}%
Therefore by \textbf{(F.1) }and\textbf{\ (a.6)} for sufficiently large $%
\left\Vert u\right\Vert $ we obtain%
\begin{equation}
\dsum\limits_{x\in S}F(x,u_{+}(x))\leq \frac{\varphi _{2}^{+}}{m_{2}^{-}}%
\left( \left\vert S\right\vert \left\Vert u\right\Vert
^{m_{2}^{+}}+\left\vert S\right\vert \right) +\psi _{2}^{+}\left\vert
S\right\vert \left\vert \overline{S}\right\vert ^{\frac{1}{2}}\left\Vert
u\right\Vert ^{m_{2}^{+}}.  \label{dir1}
\end{equation}%
\textbf{\ }By \textbf{(a.4)} and (\ref{dir1}) for sufficiently large $%
\left\Vert u\right\Vert $ since $\underset{\overline{S}}{\int }\frac{1}{p}%
\nabla _{p,\omega }u\circ \nabla _{\omega }u\geq 0$ we get%
\begin{equation*}
\begin{array}{l}
J(u)\geq \underset{S}{\int }\frac{1}{p}q\left\vert u\right\vert ^{p}-\lambda 
\underset{S}{\int }F_{u_{+}}\geq \bigskip \\ 
\frac{q^{-}}{p^{+}}\left( 2^{-\frac{p^{-}}{2}}\left\vert \partial
S\right\vert ^{\frac{p^{-}}{2}}\left\vert \overline{S}\right\vert
^{1-p^{-}}\left\Vert u\right\Vert ^{p^{-}}-\left\vert S\right\vert \right)
-\bigskip \\ 
\lambda \left( \frac{\varphi _{2}^{+}}{m_{2}^{-}}\left\vert S\right\vert
\left\Vert u\right\Vert ^{m_{2}^{+}}+\frac{\varphi _{2}^{+}}{m_{2}^{-}}%
\left\vert S\right\vert +\psi _{2}^{+}\left\vert S\right\vert \left\vert 
\overline{S}\right\vert ^{\frac{1}{2}}\left\Vert u\right\Vert
^{m_{2}^{+}}\right) .%
\end{array}%
\end{equation*}%
Since $m_{2}^{+}<p^{-},$ so $J$ is coercive on $A.$

The assumptions of Theorem \ref{mbrw} are satisfied and by Lemma \ref%
{positive} problem (\ref{uklad}) has at least one positive solution\bigskip .
\end{proof}

Put%
\begin{equation*}
\lambda _{1}:=\frac{\frac{q^{-}}{p^{+}}2^{-\frac{p^{-}}{2}}\left\vert
\partial S\right\vert ^{\frac{p^{-}}{2}}\left\vert \overline{S}\right\vert
^{1-p^{-}}}{\left( \frac{\varphi _{2}^{+}}{m_{2}^{-}}+\psi
_{2}^{+}\left\vert \overline{S}\right\vert ^{\frac{1}{2}}\right) \left\vert
S\right\vert }.\bigskip
\end{equation*}

\begin{remark}
Let $m_{2}^{+}=p^{-}.$ Assume that condition \textbf{(f.1)} is satisfied.
Then for all $\lambda \in (0,\lambda _{1})$ problem (\ref{uklad}) has at
least one positive solution.
\end{remark}

\begin{proof}
The assertion follows immediately by the proof of Theorem \ref{TheoDirect}.
\end{proof}

\subsection{Result by the Ekeland variational principle}

We have shown that the problem under consideration have at least one
positive solution for all $\lambda >0$ in case $m_{2}^{+}<p^{-}.$ In this
subsection we apply Ekeland's variational principle in order to prove the
existence of at least one positive solution for our problem for every
parameter $\lambda $ from some interval $(0,\lambda _{2})$ with no
inequality relation required on functions $m_{1}$, $m_{2}$ and $p$ apart
from the assumption that $p^{-}\neq m_{1}^{+}$ at the expense of taking a
suitable parameter interval.

Put%
\begin{equation*}
\lambda _{2}:=\frac{\frac{q^{-}}{p^{+}}2^{-\frac{p^{+}}{2}}\left\vert
\partial S\right\vert ^{\frac{p^{+}}{2}}\left\vert \overline{S}\right\vert
^{1-p^{+}}\left\vert \overline{S}\right\vert ^{-\frac{p^{+}}{2}}}{\left( 
\frac{\varphi _{2}^{+}}{m_{2}^{-}}\left\vert \overline{S}\right\vert ^{-%
\frac{m_{2}^{-}}{2}}+\psi _{2}^{+}\right) \left\vert S\right\vert }
\end{equation*}%
and 
\begin{equation*}
\Omega :=\left\{ u\in A:\left\Vert u\right\Vert \leq \left\vert \overline{S}%
\right\vert ^{-\frac{1}{2}}\right\} .
\end{equation*}

\begin{theorem}
\label{Ekeland}Let $p^{-}\neq m_{1}^{+}$. \textit{Assume that condition 
\textbf{(f.1)} is satisfied. }Then for any $\lambda \in (0,\lambda _{2})$
problem (\ref{uklad}) has at least one positive solution.
\end{theorem}

\begin{proof}
Let $\lambda \in (0,\lambda _{2})$ be fixed. For all $u\in \Omega $ by 
\textbf{(a.7) }it follows that 
\begin{equation}
\left\vert u\left( x\right) \right\vert \leq \max_{s\in S}\left\vert u\left(
s\right) \right\vert \leq \left\vert \overline{S}\right\vert ^{\frac{1}{2}%
}\left\Vert u\right\Vert \leq 1  \label{jeden}
\end{equation}%
for all $x\in S$. By \textbf{(a.1)} for all $u\in \Omega $ we get 
\begin{equation}
\dsum\limits_{x\in S}\left\vert u_{+}(x)\right\vert ^{m_{2}(x)}\leq
\dsum\limits_{x\in S}\left\vert u(x)\right\vert ^{m_{2}(x)}\leq
\dsum\limits_{x\in S}\left\vert u(x)\right\vert ^{m_{2}^{-}}\leq \left\vert
S\right\vert \left\Vert u\right\Vert ^{m_{2}^{-}}.  \label{E5}
\end{equation}%
By \textbf{(F.1), }(\ref{E5})\textbf{\ }and (\ref{u+norma}) for all $u\in
\Omega $ we see that

\begin{equation}
\dsum\limits_{x\in S}F(x,u_{+}(x))\leq \left( \frac{\varphi _{2}^{+}}{%
m_{2}^{-}}\left\Vert u\right\Vert ^{m_{2}^{-}}+\psi _{2}^{+}\left\vert 
\overline{S}\right\vert ^{\frac{1}{2}}\left\Vert u\right\Vert \right)
\left\vert S\right\vert .  \label{E6}
\end{equation}%
Moreover by \textbf{(a.3)} for all $u\in \Omega $ we obtain%
\begin{equation}
\dsum\limits_{x\in S}\left\vert u(x)\right\vert ^{p(x)}\geq
\dsum\limits_{x\in S}\left\vert u(x)\right\vert ^{p^{+}}\geq 2^{-\frac{p^{+}%
}{2}}\left\vert \partial S\right\vert ^{\frac{p^{+}}{2}}\left\vert \overline{%
S}\right\vert ^{1-p^{+}}\left\Vert u\right\Vert ^{p^{+}}  \label{E1}
\end{equation}%
Therefore for all $u\in \partial \Omega $ by (\ref{E1}) and (\ref{E6}) we
get 
\begin{equation*}
J\left( u\right) \geq \frac{q^{-}}{p^{+}}2^{-\frac{p^{+}}{2}}\left\vert
\partial S\right\vert ^{\frac{p^{+}}{2}}\left\vert \overline{S}\right\vert
^{1-p^{+}}\left\vert \overline{S}\right\vert ^{-\frac{p^{+}}{2}}-\lambda
\left( \frac{\varphi _{2}^{+}}{m_{2}^{-}}\left\vert \overline{S}\right\vert
^{-\frac{m_{2}^{-}}{2}}+\psi _{2}^{+}\right) \left\vert S\right\vert .
\end{equation*}%
Thus for all $\lambda \in (0,\lambda _{2})$ and for all$\ u\in \partial
\Omega $ we have 
\begin{equation}
J\left( u\right) >0\text{.}  \label{Jdod}
\end{equation}%
Since $\partial \Omega $ is a closed bounded set and since $J$ is
continuous, by the classical Weierstrass theorem we see that 
\begin{equation}
\underset{u\in \partial \Omega }{\inf }J(u)=\underset{u\in \partial \Omega }{%
\min }J(u)>0.  \label{EVP_1}
\end{equation}

Put 
\begin{equation*}
t_{0}:=\min \left\{ 1,\left( \frac{2\lambda \left( \frac{\varphi _{1}^{-}}{%
m_{1}^{+}}+\psi _{1}^{-}\right) p^{-}}{\overline{\omega }^{+}\left(
2\left\vert S\right\vert +\left\vert \partial S\right\vert -1\right) +2q^{+}}%
\right) ^{\frac{1}{p^{-}-m_{1}^{+}}}\right\}
\end{equation*}%
and fix $t\in (0,t_{0}).$ Let $u_{0}\in Int\Omega $ be such a function that $%
u_{0}(x_{0})=t$ and $u_{0}(x)=0$ for any $x\in S\backslash \{x_{0}\}.$ First
note that%
\begin{equation}
\begin{array}{l}
\dsum\limits_{x,y\in S}|u_{0}(y)-u_{0}(x)|^{p(x)}= \\ 
\dsum\limits_{x\in S}\left(
|u_{0}(x_{0})-u_{0}(x)|^{p(x)}+\dsum\limits_{y\neq
x_{0}}|u_{0}(y)+u_{0}(x)|^{p(x)}\right) =\bigskip \\ 
\dsum\limits_{x\in S}\left( |t-u_{0}(x)|^{p(x)}+\left( \left\vert
S\right\vert -1\right) |u_{0}(x)|^{p(x)}\right) =\bigskip \\ 
\left( \left\vert S\right\vert -1\right) t^{p(x_{0})}+\dsum\limits_{x\neq
x_{0}}t^{p(x)}\leq 2\left( \left\vert S\right\vert -1\right) t^{p^{-}}.%
\end{array}
\label{E2}
\end{equation}%
Next we can observe that%
\begin{equation}
\begin{array}{l}
\dsum\limits_{x\in \partial S}\dsum\limits_{y\in
S}|u_{0}(y)-u_{0}(x)|^{p(x)}=\bigskip \\ 
\dsum\limits_{x\in \partial S}\left( |t-u_{0}(x)|^{p(x)}+\left( \left\vert
S\right\vert -1\right) |u_{0}(x)|^{p(x)}\right) \leq \left\vert \partial
S\right\vert t^{p^{-}},%
\end{array}
\label{E3}
\end{equation}%
and 
\begin{equation}
\dsum\limits_{x\in S}\dsum\limits_{y\in \partial
S}|u_{0}(y)-u_{0}(x)|^{p(x)}=\dsum\limits_{x\in
S}|u_{0}(x)|^{p(x)}=t^{p(x_{0})}\leq t^{p^{-}}.  \label{E4}
\end{equation}%
By (\ref{E2}), (\ref{E3}) and (\ref{E4}) we get%
\begin{equation*}
\begin{array}{l}
\dsum\limits_{x,y\in \overline{S}}|u_{0}(y)-u_{0}(x)|^{p(x)}\omega (x,y)\leq 
\overline{\omega }^{+}\dsum\limits_{x,y\in
S}|u_{0}(y)-u_{0}(x)|^{p(x)}+\bigskip \\ 
\overline{\omega }^{+}\dsum\limits_{x\in \partial S}\dsum\limits_{y\in
S}|u_{0}(y)-u_{0}(x)|^{p(x)}+\overline{\omega }^{+}\dsum\limits_{x\in
S}\dsum\limits_{y\in \partial S}|u_{0}(y)-u_{0}(x)|^{p(x)}\leq \bigskip \\ 
\overline{\omega }^{+}\left( 2\left\vert S\right\vert +\left\vert \partial
S\right\vert -1\right) t^{p^{-}}.%
\end{array}%
\end{equation*}%
Therefore since $t\in (0,t_{0})$ we have%
\begin{equation*}
J(u_{0})\leq \frac{1}{2p^{-}}\overline{\omega }^{+}\left( 2\left\vert
S\right\vert +\left\vert \partial S\right\vert -1\right) t^{p^{-}}+\frac{%
q^{+}}{p^{-}}t^{p^{-}}-\lambda \left( \frac{\varphi _{1}^{-}}{m_{1}^{+}}%
+\psi _{1}^{-}\right) t^{m_{1}^{+}}<0.
\end{equation*}%
Hence 
\begin{equation}
-\infty <\underset{u\in Int\Omega }{\inf }J(u)<0.  \label{EVP_2}
\end{equation}%
By (\ref{EVP_1}) and (\ref{EVP_2}) we deduce that 
\begin{equation*}
\underset{u\in Int\Omega }{\inf }J(u)<\underset{u\in \partial \Omega }{\inf }%
J(u).
\end{equation*}%
Remaining part of the proof is based on the relevant result from \cite{art4,
MRT} but since in the source mentioned it is derived for discrete BVP we
decided to provide it in our setting for reader's convenience. Choose $%
\varepsilon >0$\ such that 
\begin{equation}
\underset{u\in \partial \Omega }{\inf }J(u)-\underset{u\in Int\Omega }{\inf }%
J(u)>\varepsilon .  \label{epsil}
\end{equation}%
Applying Ekeland's variational principle, Theorem \ref{Ekeland}, to the
functional $J:\Omega \rightarrow \mathbb{R}$ we find $u_{\varepsilon }\in
\Omega $ such that%
\begin{equation}
J(u_{\varepsilon })\leq \underset{u\in \Omega }{\inf }J(u)+\varepsilon \text{
\ \ }  \label{E_inf2}
\end{equation}%
and 
\begin{equation*}
J(u_{\varepsilon })<J(u)+\varepsilon \left\Vert u-u_{\varepsilon
}\right\Vert \text{ \ for all }u\in \Omega \text{ with }u\neq u_{\varepsilon
},\bigskip
\end{equation*}%
with $\varepsilon >0$ satisfying (\ref{epsil}). By (\ref{epsil}) and (\ref%
{E_inf2}) we get 
\begin{equation*}
J(u_{\varepsilon })\underset{u\in \Omega }{\leq \inf }J(u)+\varepsilon \leq 
\underset{u\in Int\Omega }{\inf }J(u)+\varepsilon <\underset{u\in \partial
\Omega }{\inf }J(u).
\end{equation*}%
Thus $u_{\varepsilon }\in Int\Omega .$ Note that $u_{\varepsilon }$ is an
argument of a minimum for the functional $\Phi :\Omega \rightarrow \mathbb{R}
$ defined by%
\begin{equation*}
\Phi (u):=J(u)+\varepsilon \left\Vert u-u_{\varepsilon }\right\Vert ,
\end{equation*}%
so for any $v\in \Omega $ and a small enough real positive $h$ we have%
\begin{equation*}
\frac{J(u_{\varepsilon }+hv)-J(u_{\varepsilon })}{h}+\varepsilon \left\Vert
v\right\Vert \geq 0.
\end{equation*}%
Letting $h\rightarrow 0$ we obtain 
\begin{equation*}
\left\langle J^{\prime }(u_{\varepsilon }),v\right\rangle +\varepsilon
\left\Vert v\right\Vert \geq 0.
\end{equation*}%
The above inequality holds for any $v\in \Omega ,$ so 
\begin{equation*}
\left\vert \left\langle J^{\prime }(u_{\varepsilon }),v\right\rangle
\right\vert \leq \varepsilon \left\Vert v\right\Vert .
\end{equation*}%
Finally,%
\begin{equation*}
\left\Vert J^{\prime }(u_{\varepsilon })\right\Vert =\underset{\left\Vert
v\right\Vert \leq 1}{\sup }\frac{\left\vert \left\langle J^{\prime
}(u_{\varepsilon }),v\right\rangle \right\vert }{\left\Vert v\right\Vert }%
\leq \varepsilon .
\end{equation*}%
Putting $\varepsilon =\frac{1}{n}$ for sufficiently large natural $n$, we
see that there exists a sequence $\{u_{n}\}\subset Int\Omega $ such that%
\begin{equation*}
J\left( u_{n}\right) \rightarrow \underset{u\in \Omega }{\inf }J(u)\text{ }\ 
\text{and }J^{\prime }\left( u_{n}\right) \rightarrow 0
\end{equation*}%
as $n\rightarrow \infty .$ The sequence $\{u_{n}\}$ is bounded in $A,$ so
there exists $v_{0}\in A$ such that, up to a subsequence, $\{u_{n}\}$
converges to $v_{0}$ in $A$. Thus by the continuity of $J$ and $J^{\prime }$
we have 
\begin{equation*}
J\left( v_{0}\right) =\underset{u\in \Omega }{\inf }J(u)\text{ and }%
J^{\prime }\left( v_{0}\right) =0.
\end{equation*}%
The above relations together with Theorem \ref{critical_point} and Lemma \ref%
{positive} imply that $v_{0}$ is a positive solution to (\ref{uklad}).
\end{proof}

\section{\protect\bigskip Multiple solutions}

\subsection{Application of the Ekeland variational principle and mountain
pass geometry}

The relation $m_{2}^{+}<p^{-}$ (studied by the direct variational approach)
yields $m_{1}^{-}<p^{-}$. Using the technique described in \cite%
{BerJebMawhin} in case $m_{1}^{-}>\overline{p}^{+}$ we will show that
problem (\ref{uklad}) has at least two positive solutions for every
parameter $\lambda $ from interval $(0,\lambda _{2}).$ In this case the
functional $J$ is neither coercive nor anticoercive (for the functional
defined on a finite dimensional real Banach space the coercivity implies the
Palais-Smale condition)\textit{\ }but it satisfies the Palais-Smale
condition. In \cite{BerJebMawhin} the Authors use the Ekeland Variational
Principle together with the Mountain Pass Lemma.

Let us formulate an auxiliary result which provides the Palais-Smale
condition.

\begin{lemma}
\label{(P-S)}Let $m_{1}^{-}>\overline{p}^{+}.$ Assume that condition \textbf{%
(f.1)} is satisfied. Then the functional $J$ satisfies the Palais-Smale
condition.
\end{lemma}

\begin{proof}
Assume that a sequence $\{u_{n}\}$ is such that $\{J(u_{n})\}$ is bounded
and $J^{\prime }(u_{n})\rightarrow 0$ as $n\rightarrow \infty $. Since the
space $A$ is finite dimensional, it is enough to show that $\{u_{n}\}$ is
bounded. Since $u_{n}(x)=u_{n}^{+}(x)-u_{n}^{-}(x)$ for all $n\in \mathbb{N}$
and all $x\in S$, it is enough to show that $\left\{ u_{n}^{+}\right\} $ and 
$\left\{ u_{n}^{-}\right\} $ are bounded.

Suppose that $\left\{ u_{n}^{-}\right\} $ is unbounded. Then we may assume
that there exists $N_{0}>0$ such that for all $n\geq N_{0}$ we have 
\begin{equation}
\left\Vert u_{n}^{-}\right\Vert \geq q^{-}\left\vert S\right\vert .
\label{PS-q}
\end{equation}%
Analogously as (\ref{u_u_}) we can show that 
\begin{equation}
\left( u_{+}(y)-u_{+}(x)\right) \left( u_{-}(y)-u_{-}(x)\right) \leq 0\text{
\ \ for every }\ x,y\in \overline{S}.  \label{u+u-}
\end{equation}%
Note also that 
\begin{equation}
\left\vert u_{-}(y)-u_{-}(x)\right\vert \leq \left\vert u(y)-u(x)\right\vert 
\text{ \ \ for every \ }\ x,y\in \overline{S}.  \label{u_u}
\end{equation}%
Using (\ref{u+u-}) and (\ref{u_u}) we obtain%
\begin{equation*}
\begin{array}{l}
\dsum\limits_{x,y\in \overline{S}}\left\vert u(y)-u(x)\right\vert
^{p(x)-2}(u(y)-u(x))(u_{-}(y)-u_{-}(x))=\bigskip \\ 
\dsum\limits_{x,y\in \overline{S}}\left\vert u(y)-u(x)\right\vert
^{p(x)-2}(u_{+}(y)-u_{+}(x))(u_{-}(y)-u_{-}(x))-\bigskip \\ 
\dsum\limits_{x,y\in \overline{S}}\left\vert u(y)-u(x)\right\vert
^{p(x)-2}\left( u_{-}(y)-u_{-}(x)\right) (u_{-}(y)-u_{-}(x))\leq \bigskip \\ 
\bigskip -\dsum\limits_{x,y\in \overline{S}}\left\vert u(y)-u(x)\right\vert
^{p(x)-2}\left( u_{-}(y)-u_{-}(x)\right) ^{2}\leq -\dsum\limits_{x,y\in 
\overline{S}}\left\vert u_{-}(y)-u_{-}(x)\right\vert ^{p(x)}.%
\end{array}%
\end{equation*}%
Moreover,%
\begin{equation*}
\begin{array}{l}
\dsum\limits_{x\in S}q(x)\left\vert u_{n}(x)\right\vert
^{p(x)-2}u_{n}(x)u_{n}^{-}(x)=\bigskip \\ 
\dsum\limits_{x\in S}q(x)\left\vert u_{n}(x)\right\vert
^{p(x)-2}(u_{n}^{+}(x)-u_{n}^{-}(x))u_{n}^{-}(x)=\bigskip \\ 
-\dsum\limits_{x\in S}q(x)\left\vert u_{n}(x)\right\vert
^{p(x)-2}(u_{n}^{-}(x))^{2}=-\dsum\limits_{x\in S}q(x)\left\vert
u_{n}^{-}(x)\right\vert ^{p(x)}.%
\end{array}%
\end{equation*}%
Bearing in mind (\ref{nr_poch}) the above relations lead to 
\begin{equation}
\begin{array}{l}
\left\langle J^{\prime }(u_{n}),u_{n}^{-}\right\rangle \leq \bigskip -\frac{1%
}{2}\dsum\limits_{x,y\in \overline{S}}\left\vert
u_{n}^{-}(y)-u_{n}^{-}(x)\right\vert ^{p(x)}\omega (x,y)-\text{ \ \ \ \ } \\ 
\dsum\limits_{x\in S}q(x)\left\vert u_{n}^{-}(x)\right\vert ^{p(x)}-\lambda
\dsum\limits_{x\in S}f(x,u_{n}^{+}(x))u_{n}^{-}(x)\leq \bigskip \\ 
-\dsum\limits_{x\in S}q(x)\left\vert u_{n}^{-}(x)\right\vert ^{p(x)}.%
\end{array}
\label{qq}
\end{equation}%
On the other hand by \textbf{(a.4)} we have%
\begin{equation}
\dsum\limits_{x\in S}q(x)\left\vert u_{n}^{-}(x)\right\vert ^{p(x)}\geq
q^{-}\left( 2^{-\frac{p^{-}}{2}}\left\vert \partial S\right\vert ^{\frac{%
p^{-}}{2}}\left\vert \overline{S}\right\vert ^{1-p^{-}}\left\Vert
u_{n}^{-}\right\Vert ^{p^{-}}-\left\vert S\right\vert \right) .  \label{qqq}
\end{equation}%
Thus by (\ref{qq}), (\ref{qqq}) and the Schwartz inequality\ we deduce that 
\begin{equation*}
q^{-}\left( 2^{-\frac{p^{-}}{2}}\left\vert \partial S\right\vert ^{\frac{%
p^{-}}{2}}\left\vert \overline{S}\right\vert ^{1-p^{-}}\left\Vert
u_{n}^{-}\right\Vert ^{p^{-}}-\left\vert S\right\vert \right) \leq \langle
J^{\prime }(u_{n}),-u_{n}^{-}\rangle \leq \Vert J^{\prime }(u_{n})\Vert
\cdot \Vert u_{n}^{-}\Vert .
\end{equation*}%
In a consequence by (\ref{PS-q}) we get 
\begin{equation*}
\begin{array}{l}
q^{-}2^{-\frac{p^{-}}{2}}\left\vert \partial S\right\vert ^{\frac{p^{-}}{2}%
}\left\vert \overline{S}\right\vert ^{1-p^{-}}\left\Vert
u_{n}^{-}\right\Vert ^{p^{-}}\leq \bigskip \\ 
\Vert J^{\prime }(u_{n})\Vert \cdot \Vert u_{n}^{-}\Vert +q^{-}\left\vert
S\right\vert \leq \bigskip \\ 
\Vert J^{\prime }(u_{n})\Vert \cdot \Vert u_{n}^{-}\Vert +\Vert
u_{n}^{-}\Vert \leq \left( \Vert J^{\prime }(u_{n})\Vert +1\right) \Vert
u_{n}^{-}\Vert .%
\end{array}%
\end{equation*}%
By the above, since for some fixed $\varepsilon >0$ there exists $N_{1}\geq
N_{0}$ such that $\Vert J^{\prime }(u_{n})\Vert <\varepsilon $ for every $%
n\geq N_{1}$, we get 
\begin{equation*}
\Vert u_{n}^{-}\Vert ^{p^{-}-1}\leq \frac{\left( \varepsilon +1\right) }{%
q^{-}2^{-\frac{p^{-}}{2}}\left\vert \partial S\right\vert ^{\frac{p^{-}}{2}%
}\left\vert \overline{S}\right\vert ^{1-p^{-}}}\text{.}
\end{equation*}%
Contradiction. This means that $\left\{ u_{n}^{-}\right\} $ is bounded.$%
\bigskip $

It remains to show that $\{u_{n}^{+}\}$ is bounded. Suppose that $%
\{u_{n}^{+}\}$ is unbounded. By\textbf{\ (F.1) }and \textbf{(a.4)} for
sufficiently large $\left\Vert u_{n}^{+}\right\Vert $ we obtain 
\begin{equation}
\dsum\limits_{x\in S}F(x,u_{n}^{+}(x))\geq \frac{\varphi _{1}^{-}}{m_{1}^{+}}%
\left( 2^{-\frac{m_{1}^{-}}{2}}\left\vert \partial S\right\vert ^{\frac{%
m_{1}^{-}}{2}}\left\vert \overline{S}\right\vert ^{1-m_{1}^{-}}\left\Vert
u_{n}^{+}\right\Vert ^{m_{1}^{-}}-\left\vert S\right\vert \right) +\psi
_{1}^{-}\dsum\limits_{x\in S}u_{n}^{+}(x).  \label{PS1}
\end{equation}%
By \textbf{(a.5)},{\ }\textbf{(a.6)} and (\ref{PS1}) we get 
\begin{equation*}
\begin{array}{l}
J(u_{n})\leq \frac{1}{2\overline{p}^{-}}\overline{\omega }^{+}\left( 2^{%
\overline{p}^{+}}\left\vert \overline{S}\right\vert \left\vert S\right\vert
\left\Vert u_{n}^{+}-u_{n}^{-}\right\Vert ^{\overline{p}^{+}}+\left\vert 
\overline{S}\right\vert ^{2}\right) +\frac{q^{+}}{p^{-}}\left( \left\vert
S\right\vert \left\Vert u_{n}^{+}-u_{n}^{-}\right\Vert ^{p^{+}}+\left\vert
S\right\vert \right) -\bigskip \\ 
\lambda \left( \frac{\varphi _{1}^{-}}{m_{1}^{+}}2^{-\frac{m_{1}^{-}}{2}%
}\left\vert \partial S\right\vert ^{\frac{m_{1}^{-}}{2}}\left\vert \overline{%
S}\right\vert ^{1-m_{1}^{-}}\left\Vert u_{n}^{+}\right\Vert ^{m_{1}^{-}}-%
\frac{\varphi _{1}^{-}}{m_{1}^{+}}\left\vert S\right\vert +\psi
_{1}^{-}\dsum\limits_{x\in S}u_{n}^{+}(x)\right) \leq \bigskip \\ 
\frac{1}{2\overline{p}^{-}}\overline{\omega }^{+}\left( 2^{\overline{p}%
^{+}\left( \overline{p}^{+}-1\right) }\left\vert \overline{S}\right\vert
\left\vert S\right\vert \left( \left\Vert u_{n}^{+}\right\Vert ^{\overline{p}%
^{+}}+\left\Vert u_{n}^{-}\right\Vert ^{\overline{p}^{+}}\right) +\left\vert 
\overline{S}\right\vert ^{2}\right) +\bigskip \\ 
\frac{q^{+}}{p^{-}}\left( 2^{p^{+}-1}\left\vert S\right\vert \left(
\left\Vert u_{n}^{+}\right\Vert ^{\overline{p}^{+}}+\left\Vert
u_{n}^{-}\right\Vert ^{\overline{p}^{+}}\right) +\left\vert S\right\vert
\right) -\bigskip \\ 
\lambda \left( \frac{\varphi _{1}^{-}}{m_{1}^{+}}2^{-\frac{m_{1}^{-}}{2}%
}\left\vert \partial S\right\vert ^{\frac{m_{1}^{-}}{2}}\left\vert \overline{%
S}\right\vert ^{1-m_{1}^{-}}\left\Vert u_{n}^{+}\right\Vert ^{m_{1}^{-}}-%
\frac{\varphi _{1}^{-}}{m_{1}^{+}}\left\vert S\right\vert +\psi
_{1}^{-}\dsum\limits_{x\in S}u_{n}^{+}(x)\right) .%
\end{array}%
\end{equation*}%
Since $m_{1}^{-}>\overline{p}^{+}$ {and $\{u_{n}^{+}\}$ is unbounded and $%
\{u_{n}^{-}\}$ is bounded, so }$J(u_{n})\rightarrow -\infty $ as $\Vert
u_{n}^{+}\Vert \rightarrow \infty $. Thus we obtain a contradiction with the
assumption that $\{J(u_{n})\}$ is bounded, so it follows that $\{u_{n}^{+}\}$
is bounded. Hence $\{u_{n}\}$ is bounded.
\end{proof}

At the end of Section \ref{SecPrelimResults} we indicated that results from 
\cite{BerJebMawhin} could be applied in order to get multiple solutions. Now
we are going to apply these for the problem under consideration. Recall that%
\begin{equation*}
\lambda _{2}:=\frac{\frac{q^{-}}{p^{+}}2^{-\frac{p^{+}}{2}}\left\vert
\partial S\right\vert ^{\frac{p^{+}}{2}}\left\vert \overline{S}\right\vert
^{1-p^{+}}\left\vert \overline{S}\right\vert ^{-\frac{p^{+}}{2}}}{\left( 
\frac{\varphi _{2}^{+}}{m_{2}^{-}}\left\vert \overline{S}\right\vert ^{-%
\frac{m_{2}^{-}}{2}}+\psi _{2}^{+}\right) \left\vert S\right\vert }
\end{equation*}

and 
\begin{equation*}
\Omega :=\left\{ u\in A:\left\Vert u\right\Vert \leq \left\vert \overline{S}%
\right\vert ^{-\frac{1}{2}}\right\} .
\end{equation*}

\begin{theorem}
\label{The_one_solution}Let $m_{1}^{-}>\overline{p}^{+}.$ Assume that
condition \textbf{(f.1)} is satisfied. Then for any $\lambda \in (0,\lambda
_{2})$ problem (\ref{uklad}) has at least two distinct positive solutions.
\end{theorem}

\begin{proof}
By Lemma \ref{(P-S)} the functional $J$ satisfies the Palais-Smale
condition. Let $\lambda \in (0,\lambda _{2})$ be fixed. Note that $\overline{%
Int\Omega }=\Omega .$ Since $0\in Int\Omega $, so by (\ref{EVP_1}) we deduce
that 
\begin{equation}
\min_{u\in \Omega }J(u)\leq J(0)=0<\min_{u\in \partial \Omega }J(u).
\label{J0<J_na_brzegu}
\end{equation}%
Thus we have relation (\ref{J_domX<J_brzegX}) satisfied.

Let $u_{\xi }\in A$ \ be defined as follows: $u_{\xi }(x)=\xi $ for all $%
x\in S$ and $u_{\xi }(0)=0$ for all $x\in \partial S$. Then for $\xi >1$ we
have%
\begin{equation*}
\begin{array}{l}
\dsum\limits_{x,y\in \overline{S}}|u(y)-u(x)|^{p(x)}=\bigskip \\ 
\dsum\limits_{x\in \partial S}\dsum\limits_{y\in
S}|u(y)-u(x)|^{p(x)}+\dsum\limits_{x\in S}\dsum\limits_{y\in \partial
S}|u(y)-u(x)|^{p(x)}=\bigskip \\ 
\dsum\limits_{x\in \partial S}|\xi -u(x)|^{p(x)}+\dsum\limits_{x\in
S}|u(x)|^{p(x)}\leq \left( \left\vert \partial S\right\vert +\left\vert
S\right\vert \right) \xi ^{\overline{p}^{+}}.%
\end{array}%
\end{equation*}%
Therefore%
\begin{equation*}
J(u_{\xi })\leq \frac{1}{2\overline{p}^{-}}(\left\vert \partial S\right\vert
+\left\vert S\right\vert )\left\vert \xi \right\vert ^{\overline{p}^{+}}%
\overline{\omega }^{+}+\frac{q^{+}}{p^{-}}\left\vert S\right\vert \left\vert
\xi \right\vert ^{p^{+}}-\lambda \left\vert S\right\vert \left( \frac{%
\varphi _{1}^{-}\xi ^{m_{1}^{-}}}{m_{1}^{+}}+\psi _{1}^{-}\xi \right)
\bigskip .
\end{equation*}%
Since $m_{1}^{-}>\overline{p}^{+}$, then $\lim_{\xi \rightarrow \infty
}J(u_{\xi })=-\infty $, so there exists $\xi _{0}$ such that $u_{\xi
_{0}}\in A\backslash \Omega $ and 
\begin{equation*}
J(u_{\xi _{0}})<\min_{u\in \partial \Omega }J(u).
\end{equation*}%
Thus by the remarks contained at the end of Section \ref{SecPrelimResults}
provide the assertion.
\end{proof}

Now we proceed with some suggestion of the alternative proof of Theorem \ref%
{The_one_solution} by using Corollary 3.2. from \cite{CANDITOADV} which says
that if a functional satisfying the Palais-Smale condition is unbounded from
below and has a local minimum then it has another critical point. From the
proof of Theorem \ref{The_one_solution} it follows that the functional $J$
has a local minimum on a ball%
\begin{equation*}
\Omega :=\left\{ u\in A:\left\Vert u\right\Vert \leq \left\vert \overline{S}%
\right\vert ^{-\frac{1}{2}}\right\} .
\end{equation*}%
From the proof of Lemma \ref{(P-S)} we see that the functional $J$ satisfies
the Palais-Smale condition. Moreover, from the proof of Theorem \ref%
{The_one_solution} it follows that the functional $J$ is unbounded from
below. To conclude, both methods require similar calculations to be
performed since both abstract results are based on similar tools.

\subsection{Application of the mountain pass geometry and Karush-Kuhn-Tucker
theorem}

When relation (\ref{J_domX<J_brzegX}) is not satisfied we cannot use the
arguments mentioned in Theorem \ref{The_one_solution} since this condition
is crucial since one solution is obtained via the Ekeland's Principle and it
must lie in the interior of the set, while the second one it reached through
the Mountain Pass Geometry. But we have some other tools at hand. So in this
subsection we apply Karush-Kuhn-Tucker theorem together with the mountain
pass geometry in order to obtain the existence of at least two distinct
positive solutions with at least one solution outside the unit ball. The
first minimizer we find using the Karush-Kuhn-Tucker conditions. The second
minimizer there exists by the mountain pass technique. Thus our ideas are
related to those contained in \cite{BerJebMawhin} since one solution is
reached by the mountain pass technique and the second by some other
technique which provides that it lies in the interior of the ball.

Put

\begin{equation*}
\gamma _{0}:=2^{\frac{1}{2}}\left\vert \partial S\right\vert ^{\frac{1}{2}%
}\left\vert \overline{S}\right\vert .
\end{equation*}

\begin{theorem}
\label{maintheorem}Let $m_{1}^{-}>\overline{p}^{+}.$ Assume that condition 
\textbf{(f.1)} is satisfied. Let us choose $\gamma >\gamma _{0}$ and put 
\begin{equation*}
\lambda _{3}:=\frac{q^{-}2^{-\frac{p^{-}}{2}}\left\vert \partial
S\right\vert ^{\frac{p^{-}}{2}}\left\vert \overline{S}\right\vert
^{1-p^{-}}\gamma ^{p^{-}}-q^{-}\left\vert S\right\vert }{\left( \varphi
_{2}^{+}\gamma ^{m_{2}^{+}}+\varphi _{2}^{+}+\psi _{2}^{+}\left\vert 
\overline{S}\right\vert ^{\frac{1}{2}}\gamma \right) \left\vert S\right\vert 
}.\ 
\end{equation*}%
Then for any $\lambda \in \left( 0,\lambda _{3}\right) $ problem (\ref{uklad}%
) has at least two distinct positive solutions with at least one positive
solution outside the unit ball.
\end{theorem}

\begin{proof}
Let $\lambda \in (0,\lambda _{3})$ be fixed. Note that $\gamma _{0}>1.$ Put 
\begin{equation*}
\Omega _{1}:=\left\{ u\in A:\left\Vert u\right\Vert \leq \gamma \right\} ;%
\text{ \ \ }\Omega _{2}:=\left\{ u\in A:\left\Vert u\right\Vert \geq \zeta
\right\} ,
\end{equation*}%
where $\zeta \in (1,\gamma ).$\ Put%
\begin{equation*}
\Omega :=\Omega _{1}\cap \Omega _{2}.
\end{equation*}%
\qquad The set $\Omega $ is bounded and closed, so the classical Weierstrass
theorem implies that the functional $J$ attains a minimum in $\Omega .$
Assume that $u_{0}\in A$\ is a local minimizer of $J$\ in\textit{\ }$\Omega $%
. We will show, by a contradiction, that $u_{0}$ is the element required by
the Mountain Pass Lemma, that is $u_{0}\notin \partial \Omega _{1}$. Suppose
otherwise, that $u_{0}\in \partial \Omega _{1}$.

Applying the Karush-Kuhn-Tucker theorem, Theorem \ref{KKT-THEO}, to the
problem 
\begin{equation*}
\underset{u\in A}{\min }J(u)
\end{equation*}%
subject to the constrains 
\begin{equation*}
\left\{ 
\begin{array}{c}
\left\Vert u\right\Vert ^{2}-\gamma ^{2}\leq 0; \\ 
\zeta ^{2}-\left\Vert u\right\Vert ^{2}\leq 0,%
\end{array}%
\right.
\end{equation*}%
we deduce that there exist constants $\kappa ,\sigma ,\vartheta \geq 0$ do
not vanish simultaneously such that 
\begin{equation}
\sigma (\left\Vert u_{0}\right\Vert ^{2}-\gamma ^{2})=0\text{ \ \ \ \ \ and
\ \ \ \ \ }\vartheta (\zeta ^{2}-\left\Vert u_{0}\right\Vert ^{2})=0
\label{slackness}
\end{equation}%
and%
\begin{equation}
\kappa \langle J^{\prime }(u_{0}),v\rangle +\sigma \langle u_{0},v\rangle
-\vartheta \langle u_{0},v\rangle =0\bigskip  \label{kkt}
\end{equation}%
for all $v\in A.$

The set $\left\{ u\in \Omega :\left\Vert u\right\Vert ^{2}-\gamma ^{2}\leq 0%
\text{ and }\zeta ^{2}-\left\Vert u\right\Vert ^{2}\leq 0\right\} $ has a
non-empty interior, so we may put $\kappa =1.$ By (\ref{slackness}) we
deduce that $\vartheta =0$, since $\left\Vert u_{0}\right\Vert =\gamma \neq $%
\texttt{\ }$\zeta $ and so $\zeta ^{2}-\left\Vert u_{0}\right\Vert ^{2}\neq $%
\texttt{\ }$0$. Now suppose that $\sigma >0.$ Then by (\ref{kkt}) and (\ref%
{nr_poch}) we get%
\begin{equation*}
\begin{array}{l}
\frac{1}{2}\dsum\limits_{x,y\in \overline{S}}\left\vert
u_{0}(y)-u_{0}(x)\right\vert ^{p(x)-2}(u_{0}(y)-u_{0}(x))(v(y)-v(x))\omega
(x,y)+\bigskip \\ 
\underset{x\in S}{\tsum }q(x)\left\vert u_{0}(x)\right\vert
^{p(x)-2}u_{0}(x)v(x)-\lambda \underset{x\in S}{\tsum }%
f(x,u_{0}^{+}(x))v(x)+\sigma \underset{x\in S}{\tsum }\langle u_{0}\left(
x\right) ,v\left( x\right) \rangle =0.%
\end{array}%
\end{equation*}%
for all $v\in A.$ Taking $v=u_{0}$ we see that%
\begin{equation}
\begin{array}{l}
\frac{1}{2}\dsum\limits_{x,y\in \overline{S}}\left\vert
u_{0}(y)-u_{0}(x)\right\vert ^{p(x)}\omega (x,y)+\dsum\limits_{x\in
S}q(x)\left\vert u_{0}(x)\right\vert ^{p(x)}+\sigma \left\Vert
u_{0}\right\Vert ^{2}=\bigskip \\ 
\lambda \dsum\limits_{x\in S}f(x,u_{0}^{+}(x))u_{0}(x).%
\end{array}
\label{kkt1}
\end{equation}%
Since $u_{0}\in \partial \Omega _{1},$ we see that $\left\Vert
u_{0}\right\Vert =\gamma $. Thus by \textbf{(a.4) }we obtain%
\begin{equation}
\begin{array}{l}
\frac{1}{2}\dsum\limits_{x,y\in \overline{S}}\left\vert
u_{0}(y)-u_{0}(x)\right\vert ^{p(x)}\omega (x,y)+\dsum\limits_{x\in
S}q(x)\left\vert u_{0}(x)\right\vert ^{p(x)}+\sigma \left\Vert
u_{0}\right\Vert ^{2}\geq \bigskip \\ 
q^{-}\left( 2^{-\frac{p^{-}}{2}}\left\vert \partial S\right\vert ^{\frac{%
p^{-}}{2}}\left\vert \overline{S}\right\vert ^{1-p^{-}}\gamma
^{p^{-}}-\left\vert S\right\vert \right) +\sigma \gamma ^{2}.%
\end{array}
\label{kkt2}
\end{equation}%
Note that the term on the right is positive, since $\gamma >\gamma _{0}.$

By \textbf{(a.6) }we get%
\begin{equation}
\dsum\limits_{x\in S}\left\vert u_{0}^{+}(x)\right\vert ^{m_{2}(x)}\leq
\dsum\limits_{x\in S}\left\vert u_{0}(x)\right\vert ^{m_{2}(x)}\leq
\left\vert S\right\vert \left\Vert u_{0}\right\Vert ^{m_{2}^{+}}+\left\vert
S\right\vert =\left( \gamma ^{m_{2}^{+}}+1\right) \left\vert S\right\vert .
\label{kkt5}
\end{equation}%
Moreover by (\ref{u+norma}) we obtain%
\begin{equation}
\dsum\limits_{x\in S}u_{0}^{+}(x)\leq \left\vert S\right\vert \left\vert 
\overline{S}\right\vert ^{\frac{1}{2}}\left\Vert u_{0}\right\Vert
=\left\vert S\right\vert \left\vert \overline{S}\right\vert ^{\frac{1}{2}%
}\gamma .  \label{kkt6}
\end{equation}%
By \textbf{(f.1), (\ref{kkt5}) }and (\ref{kkt6}) we infer that%
\begin{equation}
\dsum\limits_{x\in S}f(x,u_{0}^{+}(x))u_{0}(x)\leq \dsum\limits_{x\in
S}f(x,u_{0}^{+}(x))u_{0}^{+}(x)\leq \left( \varphi _{2}^{+}\gamma
^{m_{2}^{+}}+\varphi _{2}^{+}+\psi _{2}^{+}\left\vert \overline{S}%
\right\vert ^{\frac{1}{2}}\gamma \right) \left\vert S\right\vert .
\label{kkt3}
\end{equation}%
Thus by (\ref{kkt1}), (\ref{kkt2}) and (\ref{kkt3}) we get%
\begin{equation*}
q^{-}2^{-\frac{p^{-}}{2}}\left\vert \partial S\right\vert ^{\frac{p^{-}}{2}%
}\left\vert \overline{S}\right\vert ^{1-p^{-}}\gamma
^{p^{-}}-q^{-}\left\vert S\right\vert +\sigma \gamma \leq \lambda \left(
\varphi _{2}^{+}\gamma ^{m_{2}^{+}}+\varphi _{2}^{+}+\psi _{2}^{+}\left\vert 
\overline{S}\right\vert ^{\frac{1}{2}}\gamma \right) \left\vert S\right\vert
.
\end{equation*}%
{A contradiction with the assumption }$\lambda \in \left( 0,\lambda
_{3}\right) $.

Eventually $\vartheta =\sigma =0$ and $\kappa \neq 0.$ Therefore $%
u_{0}\notin \partial \Omega _{1},$ so 
\begin{equation*}
J(u_{0})<\min_{u\in \partial \Omega _{1}}J(u).
\end{equation*}

Moreover (see the proof of Theorem \ref{The_one_solution}) there exists $%
u_{\xi _{0}}\in A\setminus \Omega _{1}$ such that $J(u_{\xi
_{0}})<\min_{u\in \partial \Omega _{1}}J(u).$

By Lemma \ref{(P-S)} and Lemma \ref{lem2} we obtain a critical value\ of the
functional $J$ for some $u^{\star }\in A.$ Moreover $u_{0}$ and $u^{\star }$
are two different critical points of $J$ and therefore by Lemma \ref%
{positive} there are two distinct positive solutions to (\ref{uklad}).$%
\bigskip $
\end{proof}

\begin{remark}
We note that the closer $\gamma $ to $\gamma _{0}$ in the above theorem, the
eigenvalue interval becomes larger.
\end{remark}

\begin{example}
Let $S=\{x_{1,}x_{2},x_{3}\},$ $\partial S=\{x_{4,}x_{5},x_{6}\}$ and put $%
\omega (x_{1,}x_{2})=\omega (x_{2},x_{3})=\omega (x_{3,}x_{1})=$ $\omega
(x_{1,}x_{4})=\omega (x_{2},x_{5})=\omega (x_{3,}x_{6})=a>0.$ Let $p:%
\overline{S}\rightarrow \mathbb{[}2,+\infty ),$ $m:S\rightarrow \lbrack
2,+\infty )$ and $q,\varphi ,\psi :S\rightarrow (0,+\infty )$ are given by
the formulas%
\begin{equation*}
m(x_{i})=2i^{2};\text{ }\varphi (x_{i})=3i-1;\text{ }\psi (x_{i})=i;\text{ }%
q(x_{i})=e^{i+31}\text{\ for }i=1,2,3
\end{equation*}%
and%
\begin{equation*}
p(x_{i})=i+3\text{ for }i=1,2,...,6.\text{\ }
\end{equation*}%
We have shown in the example \ref{przyklad} that the function $f:S\times 
\mathbb{R\rightarrow R}$ given by the formula 
\begin{equation*}
f(x,t)=(t+1)^{1-e^{-t^{2}}+m(x)}\left( \frac{2}{\pi }\arctan t+\varphi
(x)\right) +\left\vert \sin t\right\vert +\psi (x)+1
\end{equation*}%
satisfies condition \textbf{(f.1)} with $m_{1}(x)=m(x),$ $m_{2}(x)=m(x)+1,$ $%
\varphi _{1}(x)=\varphi (x),$ $\varphi _{2}(x)=2^{m(x)}\left( 1+\varphi
(x)\right) ,$ $\psi _{1}(x)=\psi (x)+1$ and $\psi _{2}(x)=\psi (x)+2.$%
\newline
Note that $18=m_{2}^{+}\nless p^{-}=4,$ so assumption of Theorem \ref%
{TheoDirect} are not satisfied, but $18=m_{1}^{+}\neq p^{-}=4$ so by Theorem %
\ref{Ekeland} we have that for all $\lambda \in (0,\lambda _{2}),$ where $%
\lambda _{2}=6.\,\allowbreak 382\,3\times 10^{-4},$ problem (\ref{uklad})
has at least one positive solution.\newline
Now put $m(x_{i})=10i$ for $i=1,2,3.$ Then $10=m_{1}^{-}>\overline{p}^{+}=9$
(assumption of Theorem \ref{TheoDirect} are not satisfied) and by Theorem %
\ref{The_one_solution} problem (\ref{uklad}) has at least two distinct
positive solutions for all $\lambda \in (0,\lambda _{2}),$ where $\lambda
_{2}=0.816\,.$ Put $\gamma $ $=14.7.$ Then by Theorem \ref{maintheorem} ($%
\gamma >\gamma _{0}=\allowbreak 14.\,\allowbreak 697)$ we have that\ for all 
$\lambda \in \left( 0,\lambda _{3}\right) ,$ where $\lambda
_{3}=2.\,\allowbreak 706\,5\times 10^{-20},$ problem (\ref{uklad}) has at
least two distinct positive solutions with at least one positive solution
outside the unit ball. \newline
If we put $m(x_{i})=2\sin ^{2}\frac{\pi i}{3}\,$, we have by Theorem \ref%
{TheoDirect} the existence of at least one positive solution for problem (%
\ref{uklad}) for all $\lambda >0,$ since $3=m_{2}^{+}<p^{-}=4.$
\end{example}

\begin{tabular}{l}
Marek Galewski, Renata Wieteska \\ 
Institute of Mathematics, \\ 
Technical University of Lodz, \\ 
Wolczanska 215, 90-924 Lodz, Poland, \\ 
marek.galewski@p.lodz.pl, renata.wieteska@p.lodz.pl%
\end{tabular}

\end{document}